\documentclass[12pt]{article}

\usepackage{amsmath}               
\usepackage{amsfonts}              
\usepackage{amsthm}                
\usepackage{mathrsfs}
\usepackage[small]{caption} 
\usepackage{color}
\newtheorem{theorem}{Theorem}[section]

\newtheorem{proposition}[theorem]{Proposition}
\newtheorem{corollary}[theorem]{Corollary}
\newtheorem{definition}[theorem]{Definition}

\newcount\refno
\pdfoutput=1
\allowdisplaybreaks

\makeindex
\begin{document} 

\title{Centralizers of the Riordan Group}

\author{Tian-Xiao He$^{1}$ and  Yuanziyi Zhang$^{2}$\\
{\small 
$^{1}$Department of Mathematics}\\
 {\small Illinois Wesleyan University, Bloomington, Illinois 61702, USA}\\
 {\small $^2$ Mathematics Department}\\
{\small University of Illinois at Urbana-Champaign, Urbana, IL 61801, USA}\\
}

\date{}

\maketitle
\setcounter{page}{1}
\pagestyle{myheadings}
\markboth{T. X. He and Y. Zhang}
{Centralizers of the Riordan Group}

\begin{abstract}
In this paper, we discuss centralizers in the Riordan group. We will see that Fa\`a di Bruno's formula is an application of the Fundamental Theorem of Riordan arrays. Then the composition group of formal power series in ${\cal F}_1$ is studied to construct the centralizers of Bell type and Lagrange type Riordan arrays. Our tools are the $A$-sequences of Riordan arrays and Fa\`a di Bruno's formula. Some combinatorial explanation and discussion about related algebraic topics are also given. 
\end{abstract}

\section{Introduction}
\label{intro}
Riordan arrays are infinite, lower triangular matrices defined by the generating function of their columns. They form a group, called {\em the Riordan group} (see Shapiro, Getu, Woan and Woodson \cite{SGWW}).

More formally, let us consider the set of formal power series (f.p.s.) ${\mathcal F} = {\mathbb {K}}[\![$$z$$]\!]$ (${{\mathbb K}}={{\mathbb R}}$ or ${{\mathbb C}}$);
the \emph{order} of $f(z)  \in {\mathcal F}$, $f(z) =\sum_{k=0}^\infty f_kz^k$ ($f_k\in {{\mathbb R}}$), is the minimal number $r\in{\mathbb N}$ such that $f_r \neq 0$; ${\mathcal F}_r$ is the set of formal power series of order $r$. Let $d(z) \in {\mathcal F}_0$ and $h(z) \in {\mathcal F}_1$; the pair $(d(z) ,\,h(z) )$ defines the {\em (proper) Riordan array} $D=(d_{n,k})_{n,k\in \mbox{\scriptsize${\mathbb{N}}$}}=(d(z), h(z))$ having
  
\begin{equation}\label{Radef}
d_{n,k} = [z^n]d(z) h(z) ^k
\end{equation}
or, in other words, having $d(z) h(z)^k$ as the generating function whose coefficients make-up the
entries of column $k$. 

From the {\it fundamental theorem of Riordan arrays} (see \cite{SGWW}), 

\[
(d(z), h(z)) f(z)=d(z) f(h(z)),
\]
it is immediate to show that the usual row-by-column product of two Riordan arrays is also a Riordan array:
\begin{equation}\label{Proddef}
    (d_1(z) ,\,h_1(z) )  (d_2(z) ,\,h_2(z) ) = (d_1(z) d_2(h_1(z) ),\,h_2(h_1(z) )).
\end{equation}
The Riordan array $I = (1,\,z)$ acts as an identity for this product, that is, $(1,\,t) (d\left( z\right) ,\,h(z))=(d\left( z\right)
,\,h(z)) (1\,,z)=(d\left( z\right) ,\,h(z))$. Let $(d\left( z\right),\,h(z)) $ be a Riordan array. when its inverse is

\begin{equation}
(d\left( z\right) ,\,h(z))^{-1}=\left( \frac{1}{d(\overline{{h}}(z))},
\overline{h}(z),\right)  \label{Invdef}
\end{equation}
where $\overline {h}(z)$ is the compositional inverse of $h(z)$, i.e., $(h\circ 
\overline{h})(z)=(\overline{h}\circ h)(z)=z$. In this way, the set $\mathcal{R}$ of 
proper Riordan arrays forms a group (see \cite{SGWW}). 

Here is a list of six important subgroups of the Riordan group (see \cite{SGWW}), where 
$d\in{\cal F}_0$ and $h\in{\cal F}_1$.

\begin{itemize}
\item the {\it Appell subgroup} ${\cal A}=\{ (d(z),\,z)\}$.
\item the {\it Lagrange (associated) subgroup} ${\cal L}=\{(1,\,h(z))\}$.
\item the {\it Bell subgroup} ${\cal B}=\{(d(z),\, zd(z))\}$. 
\item the {\it hitting-time subgroup} ${\cal H}=\{(zh'(z)/h(z),\, h(z))\}$.
\item the {\it derivative subgroup} ${\cal D}=\{ (h'(z), \, h(z))\}.$
\item the {\it checkerboard subgroup} ${\cal C}=\{ (d(z),\, h(z))\},$ where $d$ is an even function and $h$ is an odd function. 
\end{itemize}

An infinite lower triangular array $[d_{n,k}]_{n,k\in{{\mathbb N}}}=(d(z), h(z))$ is a Riordan array if and only if a unique sequence $A=(a_0\not= 0, a_1, a_2,\ldots)$ exists such that for every $n,k\in{{\mathbb N}}$  (see Merlini, Rogers, Sprugnoli, and Verri \cite {MRSV})

\begin{equation}\label{eq:1.1}
d_{n+1,k+1} =a_0 d_{n,k}+a_1d_{n,k+1}+\cdots +a_nd_{n,n}. 
\end{equation} 
This is equivalent to 
\begin{equation}\label{eq:1.2}
h(z)=zA(h(z))\quad \text{or}\quad z=\bar h(z) A(z).
\end{equation}
Here, $A(z)$ is the generating function of the $A$-sequence. There exists a unique sequence $Z=(z_0, z_1,z_2,\ldots)$ such that every element in column $0$ can be expressed as the linear combination 
\begin{equation}\label{eq:1.3}
d_{n+1,0}=z_0 d_{n,0}+z_1d_{n,1}+\cdots +z_n d_{n,n},
\end{equation}
or equivalently,
\begin{equation}\label{eq:1.4}
d(z)=\frac{1}{1-zZ(h(z))}.
\end{equation}

\begin{definition}
The centralizer of a subset $S$ of group $G$ is defined to be 
\[	
C_G (S)=\{g \in G: gs=sg\mbox{ for all} \,\,s \in S \}
\]
\end{definition}
It is worth noting that $S$ may be a singleton set. Thus, the centralizer of a
subset $S$ of the group $G$ is the set of all elements in $G$ that fix each
element of $S$ under conjugation, i.e. commute with each element in $S.$ 

\noindent {\bf Example 1.1} Let $S=\{ (1+z^2, z)\}$. For any $(d,h)\in{{\cal R}}$ in $C_{{\cal R}} (S)$, we have 

\[
(d,h)(1+z^2,z)=(1+z^2,z)(d,h),
\]
i.e., 

\[
(d(1+h^2),h)=((1+z^2)d,h),
\]
which implies $h^2=z^2$, or equivalently, $h=\pm z$. Thus, $C_{{\cal R}} (S)=\{ (d,h)\in {{\cal R}}: h=\pm z\},$ the double Appell subgroup $A^{\pm}=(d(z), \pm z)$  of the Riordan group ${{\cal R}}$. To guarantee $C_{{\cal R}}(S)$ contains only elements of ${{\cal A}}$, the Appell subgroup of ${{\cal R}}$, we need an additional condition, which will be presented in Theorem \ref{thm:4.5}. 

A \emph{generalized Riordan array} with respect to the sequence $(c_n)_{n\in\mathbb{N}}$ ($c_n\not= 0$) is a pair $(g(z), f(z))$ of formal power series, where $g(z)=\sum_{k=0}^\infty g_kz^k/c_k$ and
$f(z)=\sum_{k=1}^\infty f_kz^k/c_k$ with $f_1\neq 0$. The Riordan array $(g(z),f(z))$ defines an infinite, lower triangular array $(d_{n,k})_{0\leq k\leq n<\infty}$ according to the rule:
\begin{equation}\label{RAelement}
d_{n,k}=\left[{\frac{z^n}{c_n}}\right]g(z)\frac{(f(z))^k}{c_k},
\end{equation}
where the functions $g(z)(f(z))^k/c_k$ are column generating functions of the Riordan array.

Let $D=(g(z),f(z))=(d_{n,k})_{n,k\in\mathbb{N}}$ be a Riordan array with respect to $(c_n)_{n\in\mathbb{N}}$ and let $h(z)=\sum_{k=0}^\infty h_kz^k/c_k$ be the generating function of the sequence $(h_n)_{n\in\mathbb{N}}$. Then we have the fundamental theorem of Riordan arrays (FTRA)

\begin{equation}\label{RAbasicsumrule}
\sum_{k=0}^nd_{n,k}h_k=\left[{\frac{z^n}{c_n}}\right]g(z)h(f(z))\,,
\end{equation}
or equivalently, $(g(z),f(z))*h(z)=g(z)h(f(z))$\,.

With every formal power series $f(z)=\sum_{k=1}^\infty f_kz^k/c_k$, we
associate the following infinite lower \emph{iteration matrix} with respect to $(c_n)_{n\in\mathbb{N}}$ (see p. 145 of Comtet \cite{comtet}):
\[
B(f(z)):=\left(\begin{array}{ccccccc}
1&0&0&0&0&0&\cdots\\
0&B_{1,1}&0&0&0&0&\cdots\\
0&B_{2,1}&B_{2,2}&0&0&0&\cdots\\
0&B_{3,1}&B_{3,2}&B_{3,3}&0&0&\cdots\\
0&B_{4,1}&B_{4,2}&B_{4,3}&B_{4,4}&0&\cdots\\
0&B_{5,1}&B_{5,2}&B_{5,3}&B_{5,4}&B_{5,5}&\cdots\\
\vdots&\vdots&\vdots&\vdots&\vdots&\vdots&\ddots
\end{array}\right)\,,
\]
where $B_{n,k}=B_{n,k}(f_1,f_2,\ldots)$ is the Bell polynomial with respect to $(c_n)_{n\in\mathbb{N}}$, defined as follows:

\begin{equation}\label{Bellptoom}
\frac{1}{c_k}(f(z))^k=\sum_{n=k}^\infty B_{n,k}\frac{z^n}{c_n}\,.
\end{equation}
The most obvious choices for $(c_{n})$ are $(1)$, $(n!)$, and $(n!n!)$ or $(n!!)$.

Therefore, $B_{n,k}=[z^n/c_n](f(z))^k/c_k$, which implies that the iteration matrix $B(f(z))$ is the Riordan array $(1, f(z))$. Now, the following important property of the iteration matrix
(see Theorem A on p. 145 of Comtet \cite{comtet}, Roman \cite{Roman84}, and Roman and Rota \cite{RomanRota78} )

\[
B(f(g(z)))=B(g(z))*B(f(z))
\]
is trivial in the context of the theory of Riordan arrays, i.e.,

\[
(1,f(g(z)))=(1,g(z))*(1,f(z))\,.
\]
In the case where $c_n=n!$ we obtain that the following well-known Fa\`{a} di Bruno's formula is an application of the FTRA:

\[
\sum_{k=1}^nB_{n,k}(g_1,g_2,\ldots,g_{n-k+1})f_k
    =\left[{\frac{t^n}{n!}}\right]f(g(t))\,.
\]
More precisely, Let $f$ and $g$ be two formal (Taylor) series:

\[
f=\sum_{k\geq 0}f_{k}\frac{z^{k}}{k!}, \quad g=\sum_{k\geq 0}g_{k}\frac{z^{k}}{k!}, 
\]
where $g_{0}=0$, and let $h=f\circ g$ with formal Taylor expansion 

\[
h=f\circ g=\sum_{n\geq 0}h_{n}\frac{z^{n}}{n!}.
\]
Then the coefficients $h_n$ are given by the following expression (see Theorem A on p. 137 of  Comtet \cite{comtet}): 

\[
h_{0}=f_{0},\quad h_{n}=\sum^{n}_{k=1}f_{k}B_{n,k}(g_{1},g_{2},\ldots, g_{n-k+1}),
\]
where the $B_{n,k}$ are the exponential Bell polynomials. $B_{n,k}$ can be presented as 

\[
B_{n,k}(g_{1},g_{2},\ldots, g_{n-k+1})=\sum \frac{n!}{c_{1}!c_{2}!\cdots (1!)^{c_{1}}(2!)^{c_{2}}\cdots} x_{1}^{c_{1}}x_{2}^{c_{2}}\cdots,
\]
where the summation takes place over all integers $c_{1}$, $c_{2}$, $c_{3}, \cdots \geq 0$ such that there hold (the partition of $n$ with $k$ parts)

\[
c_{1}+2c_{2}+3c_{3}+\cdots =n, \quad c_{1}+c_{2}+c_{3}+\cdots =k.
\]

\medbreak

\noindent{\bf Example 1.2} Consider the problem of finding the $n$-th derivative of $F(x)=x^{\alpha x}$, where $x>0$ and $\alpha\not= 0$ is a real number (cf. Comtet \cite{comtet}). By using the Stirling numbers of the first kind, we expand $f(t)=F(x+t)$ as 

\begin{eqnarray*}
&&f(t)=(x+t)^{\alpha (x+t)}=x^{\alpha (x+t)}\left( 1+\frac{t}{x}\right)^{\alpha (x+t)}=x^{\alpha x}x^{\alpha t}\left( 1+\frac{t}{x}\right)^{\alpha x\left( 1+\frac{t}{x}\right)}\\
&=& F(x)\exp (\alpha t \log x)\exp\left( ax\left(1+\frac{t}{x}\right) \log \left( 1+\frac{t}{x}\right)\right).
\end{eqnarray*}
Notice that there is a well-known expansion on P. 139 of Comtet \cite{comtet} shown as 

\[
\frac{1}{k!}((1+z)\log (1+z))^{k}=\sum_{n\geq k}b(n,k)\frac{z^{n}}{n!},
\]
where $b(0,0)=1$ and 

\[
b(n,k)=\sum_{k\leq \ell\leq n}\binom{\ell}{k}k^{\ell-k}s(n,\ell)
\]
with the Stirling numbers $s(n,\ell)$ of the first kind. The numbers $b(n,k)$, called Lehmer-Comtet numbers of first kind, follow the recurrence relation 

\[
b(n+1,k)=(k-n)b(n,k)+b(n,k-1)+nb(n-1,k-1),
\]
and the first few values are listed below 
\[
\left(\begin{array}{cccccccc}
1&0&0&0&0&0&0&\cdots\\
0&1&0&0&0&0&0&\cdots\\
0&1&1&0&0&0&0&\cdots\\
0&-1&3&1&0&0&0&\cdots\\
0&2&-1&6&1&0&0&\cdots\\
0&-6&0&5&10&1&0&\cdots\\
0&24&4&-15&25&15&1&\cdots\\
\vdots&\vdots&\vdots&\vdots&\vdots&\vdots&\vdots&\ddots
\end{array}\right)\,.
\]
More properties of Lehmer-Comtet numbers of first kind can be found in OEIS A008296, P. 139 of \cite{comtet}, and Gould \cite{Gou}. However, the combinatorial explanation seems not to be known. By using Lehmer-Comtet numbers of first kind, we obtain consequently
\[
f(t)=F(x)\sum_{j\geq 0, k\leq m}b(m,k) \frac{(\alpha t\log x)^{j}(t/x)^{m}}{j!m!}(\alpha x)^{k},
\]
which implies 

\[
f_{n}=\frac{d^{n}(x^{\alpha x})}{dx^{n}}=\alpha^{n}x^{\alpha x}\sum^{n}_{j=0}\left[ \binom{n}{j}u^{j}\sum^{n-j}_{k=0}b(n-j, n-k-j)v^{k}\right],
\]
where $u=\log x$ and $v=(\alpha x)^{-1}$. For instance, by using the above table of $b(n.k)$ we get

\begin{eqnarray*}
&&f_{4}=\frac{d^{4}(x^{\alpha x})}{dx^{4}}\\
&=&\alpha^{4}x^{\alpha x}\sum^{4}_{j=0}\left[ \binom{4}{j}u^{j}\sum^{4-j}_{k=0}b(4-j, 4-k-j)v^{k}\right]\\
&=&\alpha^{4}x^{\alpha x}\left[ \sum^{4}_{k=0}b(4, 4-k)v^{k}+4u\sum^{3}_{k=0}b(3, 3-k)v^{k}\right.\\
&&\quad \left. +6u^2\sum^{2}_{k=0}b(2, 2-k)v^{k}+3u^3\sum^{1}_{k=0}b(1, 1-k)v^{k}+u^4\right]\\
&=&\alpha^{4}x^{\alpha x}\left[ 1+6v-v^{2}+2v^{3}+4u(1+3v-v^{2})+6u^{2}(1+v)+4u^{3}+u^{4}\right].
\end{eqnarray*}

Additionally, for any $f\in{\cal F}_1$, the inverse array $(1,\bar{f})=B(\bar{f})$ is also an iteration matrix. Thus, the set of iteration matrices with respect to $(c_n)_{n\in\mathbb{N}}$, denoted by $\mathcal{B}$, is a nonempty subset of the Riordan group $\mathcal{R}$ with respect to $(c_n)_{n\in\mathbb{N}}$, closed under multiplication and taking inverses in
$\mathcal{R}$. These indicate that $\mathcal{B}$ is a subgroup of $\mathcal{R}$ and we call it the \emph{associated subgroup}.

If
${\displaystyle f(x)=\sum _{n=0}^{\infty }{\frac {b_{n}}{n!}}x^{n}}$ and
${\displaystyle g(x)=\sum _{n=1}^{\infty }{\frac {a_{n}}{n!}}x^{n}}$ 
and
$ {\displaystyle f(g(x))=\sum _{n=0}^{\infty }{\frac {d_{n}}{n!}}x^{n},}$ 
then the coefficient $d_n$ (which would be the nth derivative of $f\circ g$ evaluated at $0$ if we were dealing with convergent series rather than formal power series) is given by

\[
{\displaystyle d_{n}=\sum _{\pi =\left\{B_{1},\ldots ,B_{k}\right\}}a_{\left|B_{1}\right|}\cdots a_{\left|B_{k}\right|}b_{k}}
\]
where $\pi$ runs through the set of all partitions of the set $\{1, ..., n\}$ and $B_1$, $\ldots$, $B_k$ are the blocks of the partition $\pi$, and $|B_j|$ is the number of members of the $j$th block, for $j = 1, \ldots, k.$ This version of the formula is equivalent to Fa\'a di Bruno's formula for exponential formal power series shown before, i.e.,

\[
{\displaystyle f(g(x))=b_{0}+\sum _{n=1}^{\infty }{\frac {\sum _{k=1}^{n}b_{k}B_{n,k}(a_{1},\ldots ,a_{n-k+1})}{n!}}x^{n},}
\]
where $B_{n,k}(a_1,\ldots,a_{n-k+1})$ are Bell polynomials. Fa\'a di Bruno's formula is useful in finding centralizers, commutative compositions, Riordan involutions and pseudo-involutions. The centralizers related to the Riordan pseudo-involution multiplication and palindromes are shown in a recent paper by Shapiro and one of the authors \cite{HS20}.

In the next section, we will use Fa\'a di Bruno's formula and sequence characterization of Riordan arrays to construct centralizers of Bell type Riordan arrays and Riordan arrays of the form $(1/(1-kz), f)$, $f\in{{\mathcal F}_1}$. Some combinatorial explanation are also given. In Section $3$, we will study the method of the construction of centralizers of the Riordan arrays related to the composition group of formal power series in ${\mathcal F}_1$ and its application for Lagrange type Riordan arrays and Bell type Riordan arrays.  Section $4$ diverts to some related topics such as reversers, conjugates, commutators, twisted subgroup, and pseudo-Riordan involutions.

\section{Centralizers of the Riordan group}\label{sec:1}

Suppose $f(x)=\sum _{n=0}^{\infty }{a_{n}}x^{n}$ and $g(x)=\sum _{n=0}^{\infty }{b_{n}}x^{n}$ are formal power series and $b_{0}=0.$ Then the composition $ f\circ g$ is again a formal power series,

\begin{equation}\label{0.0}
f(g(x))=\sum _{n=0}^{\infty }{c_{n}}x^{n},
\end{equation}
where $c_0 = a_0$ and the other coefficients $c_n$ for $n \geq 1$ can be expressed as a sum over compositions of $n$ or as an equivalent sum over partitions of $n$:

\[
c_{n}=\sum _{\mathbf {i} \in {\mathcal {C}}_{n}}a_{k}b_{i_{1}}b_{i_{2}}\cdots b_{i_{k}},
\]
where

\[
{\mathcal {C}}_{n}=\{(i_{1},i_{2},\dots ,i_{k})\,:\ 1\leq k\leq n,\ i_{1}+i_{2}+\cdots +i_{k}=n\}
\]
is the set of compositions of n with k denoting the number of parts, or

\begin{equation}\label{0.0-2}
c_{n}=\sum _{k=1}^{n}a_{k}\sum _{\mathbf {\pi } \in {\mathcal {P}}_{n,k}}{\binom {k}{\pi _{1},\pi _{2},...,\pi _{n}}}b_{1}^{\pi _{1}}b_{2}^{\pi _{2}}\cdots b_{n}^{\pi _{n}},
\end{equation}
where

\[
{\mathcal {P}}_{n,k}=\{(\pi _{1},\pi _{2},\dots ,\pi _{n})\,:\ \pi _{1}+\pi _{2}+\cdots +\pi _{n}=k,\ \pi _{1}\cdot 1+\pi _{2}\cdot 2+\cdots +\pi _{n}\cdot n=n\}
\]
is the set of partitions of n into k parts, in frequency-of-parts form.

The first form is obtained by picking out the coefficient of $x^n$ in $(b_{1}x+b_{2}x^{2}+\cdots )^{k}$ "by inspection", and the second form is then obtained by collecting like terms, or alternatively, by applying the multinomial theorem.

The product of two Riordan arrays is a Riordan array. Sprugnoli and one of the authors \cite{HS} study the problem of how do the $A$-sequence and $Z$-sequence of the product depend on the analogous sequences of the two factors. More precisely, consider two proper Riordan arrays  $D_1=(d_1,h_1)$ and $D_2=(d_2,h_2)$ and their product, 

\[
D_3=D_1 D_2=(d_1d_2(h_1), h_2(h_1)).
\]

\begin{theorem}\label{thm:0.1}\cite{HS}
Let $A_i(z)$, $i=1,2,$ and $3$, be the generating functions of $A$-sequences of $D_i$, $i=1,2,$ and $3$, respectively. Then 

\[
A_3(z)=A_2(z)A_1\left( \frac{z}{A_2(z)}\right).
\]
\end{theorem}

Let $h\in {\cal F}_1$. We define 

\begin{equation}\label{0.1}
A_h(z)=\frac{z}{\bar h(z)}
\end{equation}
where $\bar h$ is the compositional inverse of $h$. Thus, from Theorem \ref{thm:0.1}, we have the following result.

\begin{theorem}\label{thm:0.2}
Let $A_h$ be defined by \eqref{0.1}, and let $h,f\in {\cal F}_1$. Then 

\[
A_{h\circ f}(z)=A_f(z)A_h\left( \frac{z}{A_f(z)}\right).
\]
\end{theorem}

\begin{corollary}\label{cor:0.3}
Let $A_h$ be defined by \eqref{0.1}, and let $h,f\in {\cal F}_1$. Then $h\circ f=f\circ h$ if and only if 

\begin{equation}\label{0.2}
A_f(z)A_h\left( \frac{z}{A_f(z)}\right)=A_h(z)A_f\left( \frac{z}{A_h(z)}\right).
\end{equation}
Particularly, $h\circ f=f\circ h=z$ if and only if 

\begin{equation}\label{0.3}
A_f(z)A_h\left( \frac{z}{A_f(z)}\right)=A_h(z)A_f\left( \frac{z}{A_h(z)}\right)=1.
\end{equation}
\end{corollary}

\begin{proposition}\label{pro:0.8}
Let $B$ be the Bell subgroup of the Riordan group, and let $(h/z,h)\in B$. Then the centralizer of $(h/z,h)$ in the subgroup $B$ is 

\[
C_B((h/z,h))=\{ (f/z,f)\in B: h\circ f=f\circ h\}.
\]
\end{proposition}

\begin{proof}
Let $(f/z,f)\in B$ with $h\circ f=f\circ h$. Then 

\[
(h/z,h)(f/z,f)=((f\circ h)/z,f\circ h)=((h\circ f)/z, h\circ f)=(f/z,f)(h/z,h),
\]
completing the proof.
\end{proof}

\begin{theorem}\label{thm:0.9}
Let $S=\{(1/(1-z), z/(1-z))\}$, and let $B$ be the Bell subgroup of the Riordan group. Then the centralizer 

\begin{equation}\label{0.4}
C_B(S)=\{(1/(1-rz), z/(1-rz)): r\in {{\mathbb R}}\}.
\end{equation}
\end{theorem}

\begin{proof}
From Proposition \ref{pro:0.8}, to obtain \eqref{0.2}, it is sufficient to prove the set of all elements in ${\cal F}_1$ that commute  with $h=z/(1-z)$ is 

\begin{equation}\label{0.5}
E=\{ z/(1-rz): r\in {{\mathbb R}}\}.
\end{equation}
Let $f\in{\cal F}_1$, and let $A_f$ and  $A_h$ be defined by \eqref{0.1}. Since $\bar h=z/(1+z)$, $A_h(z)=(1+z)$.  From Corollary \ref{cor:0.3}, $h\circ f=f\circ h$ if and only if \eqref{0.2} holds, i.e., 

\[
A_f(z)\left( 1+\frac{z}{A_f(z)}\right)=(1+z)A_f\left( \frac{z}{1+z}\right),
\]
or equivalently, 

\[
A_f(z)+z=(1+z)A_f\left( \frac{z}{1+z}\right).
\]
Substituting $z=z/(1-z)$ into above equation yields 

\begin{equation}\label{0.6}
(1-z)A_f\left( \frac{z}{1-z}\right)+z=A_f(z).
\end{equation}
Denote $A_f(z)=\sum_{k\geq 0} a_kz^k$ and substitute it into \eqref{0.6}: 

\[
(1-z)(a_0+a_1(z+z^2+\cdots)+a_2(z+z^2+\cdots)^2+\cdots)=a_0+a_1z+a_2z^2+\cdots.
\]
By comparing the coefficients of  powers $z^0$, $z^1$, and $z^2$ on both sides of the resulting equation, we obtain $a_0=a_0$, 

\[
-a_0+a_1+1=a_1\quad \mbox{and} \quad a_1-a_1+a_2=a_2,
\]
which implies $a_0=1$ and $a_1$ is free. To determine $a_n$ for $n\geq 2$, we use the Lagrange Inversion Formula 

\[
[z^n]h(g(z))=\frac{1}{n}[z^{-1}]\frac{h'(z)}{\bar g^n}
\]
for $n\geq 3$ to $A_f(z/(1-z))$ in \eqref{0.6}, where $g\in {\cal F}_1$ and $\bar g$ is the compositional inversion of $g$. Then we can write the coefficients of $z^n$ on the lefthand side of \eqref{0.6} as

\begin{eqnarray}\label{0.8}
&&[z^n] A_f \left( \frac{z}{1-z}\right) -[z^{n-1}]A_f\left( \frac{z}{1-z}\right)\nonumber \\
&=&\frac{1}{n}[z^{-1}]\frac{A_f'(Z)}{\left(\frac{z}{1+z}\right)^n}-\frac{1}{n-1}[z^{-1}]\frac{A_f'(Z)}{\left( \frac{z}{1+z}\right)^{n-1}}\nonumber\\
&=&\frac{1}{n}[z^{n-1}](1+z)^nA_f'(z)-\frac{1}{n-1}[z^{n-2}](1+z)^{n-1}A_f'(Z)\nonumber\\
&=&\frac{1}{n}\sum^{n-1}_{k=0}\binom{n}{k}(n-k)a_{n-k}-\frac{1}{n-1}\sum^{n-2}_{k=0}\binom{n-1}{k}(n-k-1)a_{n-k-1}\nonumber\\
&=&\sum^{n-1}_{k=0}\binom{n-1}{k}a_{n-k}-\sum^{n-2}_{k=0}\binom{n-2}{k}a_{n-k-1}\nonumber\\
&=&a_n+\sum^{n-1}_{k=1}\left( \binom{n-1}{k}-\binom{n-2}{k-1}\right) a_{n-k}\nonumber\\
&=&a_n+\sum^{n-2}_{k=1}\binom{n-2}{k}a_{n-k}.
\end{eqnarray}
Comparing the coefficients of $z^n$ on both sides of \eqref{0.6}, we have 

\[
a_n+\sum^{n-2}_{k=1}\binom{n-2}{k}a_{n-k}=a_n,
\]
or equivalently,

\[
\sum^{n-2}_{k=1}\binom{n-2}{k}a_{n-k}=0.
\]
Using the above equation and noting $a_2=0$ when $n=3$, we may prove all $a_n=0$ for $n\geq 2$ inductively. Thus $A_f(t)=1+a_1t$ and $f(t)=t/(1-rt)$ for $r=a_1\in {{\mathbb R}}$. 
\end{proof}

We now give a combinatorial proof of Theorem \ref{thm:0.9} based on a combinatorial interpretation of expression of \eqref{0.0-2}. 

\begin{theorem}\label{thm:0.11}
Let $f(x)=\sum _{n=0}^{\infty }{a_{n}}x^{n}$ and $g(x)=\sum _{n=0}^{\infty }{b_{n}}x^{n}$ be formal power series and $b_{0}=0$, and let $f\circ g$ be the composition of $f$ and $g$ shown in \eqref{0.0}, $f(g(x))=\sum _{n=0}^{\infty }{c_{n}}x^{n}$, where $c_0 = a_0$ and the other coefficients $c_n$ for $n \geq 1$ are expressed as a sum over compositions of $n$ or as an equivalent sum over partitions of $n$ shown in \eqref{0.0-2}. Then 

\begin{equation}\label{0.9}
c_{n}=\sum_{k=1}^{n} B_{n,k} a_k,
\end{equation}
where $B_{n,k}=\sum b_{j_1}b_{j_2}...b_{j_k}$, where $|j|=j_1+j_2+...+j_k=n$ and $B_{n,k}$ could be interpreted as the combinations of ways to put $n$ non-distinguished objects into $k$ nonempty distinguished (labeled) boxes. 
\end{theorem}

\begin{proof}
As an example of $B_{n,k}$, $B_{2,2}={b_1}^2$ can be interpreted as combinations of ways to put $2$ objects into $2$ boxes. Similarly, $B_{4,2}=2b_1b_3+{b_2}^2$.  We now give a combinatorial proof of Theorem \ref{thm:0.9}. Let $h(t)=\frac{1}{1-t}=\sum_{k=1}^{\infty}a_kt^k$ where $a_k=1$ for all $k \in \mathbb{N}$. We want to find elements $f$ of $C_{B}(h)$, $f(t)=\sum_{k=1}^{\infty} b_kt^k$ where $b_1=1$ such that $h(f)=f(h)$. 
	Let $h(f)=\sum_{k=1}^{\infty} c_kt^k$ and $f(h)=\sum_{k=1}^{\infty} \hat{c}_kt^k$. We first compute the first few terms.
	\bigbreak
	For $k=1$, we have
	\begin{equation*}
		c_1=B_{1,1}a_1=b_1a_1.
	\end{equation*}
	Switch the terms accordingly, we obtain 
	\begin{equation*}
		\hat{c}_1=a_1b_1.
	\end{equation*}
	Then $c_1=\hat{c}_1$ naturally.
	\smallbreak
	For $k=2$, we have
	\begin{equation*}
	c_2=B_{2,1}a_1+B_{2,2}a_2=b_2a_1+{b_1}^2a_2,
	\end{equation*}
	\begin{equation*}
	\hat{c}_2=a_2b_1+{a_1}^2b_2.
	\end{equation*}
	Plugging in $a_k=1$ and $b_1=1$, and then equate $c_2$ and $\hat{c}_2$ we have
	\begin{equation*}
	b_2+1=1+b_2.
	\end{equation*}
	Therefore, we conclude that $b_2$ is free.
	\smallbreak
	For $k=3$, we have
	\begin{equation*}
	\begin{split}
	c_3 & = B_{3,1}a_1+B_{3,2}a_2+B_{3,3}a_3 \\
	& = a_1b_3+a_2(b_1b_2+b_2b_1)+a_3{b_1}^3 \\
	& = a_1b_3+2a_2b_1b_2+a_3{b_1}^3,
	\end{split}
	\end{equation*}
	\begin{equation*}
	\hat{c}_3=b_1a_3+2b_2a_1a_2+b_3{a_1}^3.
	\end{equation*}
	Equate $c_3$ and $\hat{c}_3$, the equation holds naturally.
	\smallbreak
	For $k=4$, we have
	\begin{equation*}
	\begin{split}
	c_4 & = B_{4,1}a_1+B_{4,2}a_2+B_{4,3}a_3+B_{4,4}a_4 \\
	& = a_1b_4+a_2(2b_1b_3+{b_2}^2)+a_3(3b_1b_1b_2)+a_4{b_1}^4 \\
	& = a_1b_4+2a_2b_1b_3+a_2{b_2}^2+3a_3b_1b_1b_2+a_4{b_1}^4,
	\end{split}
	\end{equation*}
	\begin{equation*}
	\hat{c}_4=b_1a_4+2b_2a_1a_3+b_2{a_2}^2+3b_3a_1a_1a_2+b_4{a_1}^4.
	\end{equation*}
	Plugging in $a_k=1$ and $b_1=1$, and then equate $c_4$ and $\hat{c}_4$ we have
	\begin{equation*}
	b_4+2b_3+{b_2}^2+3b_2+1=1+2b_2+b_2+3b_3+b_4,
	\end{equation*}
	\begin{equation*}
	2b_3+{b_2}^2=3b_3.
	\end{equation*}
	Therefore, we conclude that $b_3={b_2}^2$.
	\smallbreak
	For $k=5$, we have
	\begin{equation*}
	\begin{split}
	c_5 & = B_{5,1}a_1+B_{5,2}a_2+B_{5,3}a_3+B_{5,4}a_4+B_{5,5}a_5 \\
	& = a_1b_5+a_2(2b_1b_4+2b_2b_3)+a_3({3b_1}^2b_3+3b_1{b_2}^2)+a_4(4{b_1}^3b_2)+a_5{b_1}^5,
	\end{split}
	\end{equation*}
	\begin{equation*}
	\hat{c}_5=b_1a_5+b_2(2a_1a_4+2a_2a_3)+b_3(3{a_1}^2a_3+3a_1{a_2}^2)+b_4(4{a_1}^3a_2)+b_5{a_1}^5.
	\end{equation*}
	Equate $c_5$ and $\hat{c}_5$ we have
	\begin{equation*}
	2b_4+2b_2b_3+3b_3+3{b_2}^2=6b_3+4b_4.
	\end{equation*}
	Therefore, we conclude that $b_4={b_2}^3$.
	
	\bigbreak
	Base on the computational results above, we make the conjecture that $b_p={b_2}^{p-1}$ for all $p \in \mathbb{N}$ and $p \geq 2$. We prove this result by induction.

For $p=1$, the result holds naturally as demonstrated above. Suppose now the result holds up to $p=q$. Since $a_k=1$ for all $k \in \mathbb{N}$, as a matter of simplicity, we put the coefficients $B_{n,k}$ as the entries of a lower triangular matrix below,
	
	\begin{equation*}
	\begin{pmatrix} 
	1   \\
	b_2 & 1 \\
	{b_2}^2 & 2b_2 & 1 \\
	{b_2}^3 & 3{b_2}^2 & 3b_2 & 1 \\
	\vdots &  &  &  & \ddots \\
	{b_2}^q & q{b_2}^{q-1} & \cdots  & & \cdots & 1 \\
	\end{pmatrix}.
	\end{equation*}
	
Observe that the matrix above is a Riordan array with A-sequence $a_0=1, a_1=b_2$.

Now we want to show that the result holds for $p=q+1$. It is sufficient to show that adding another row and column to the matrix above, we still obtain a Riordan array preserving the same A-sequence. First, we consider the combinatorial interpretation of $B_{q+1,k}$ for $k \geq 3$. If we have $k$ boxes and $q+1$ objects, each of the boxes contains at most $q-k+2$ objects, i.e., the subscript of the terms that appear in $B_{q+1,k}$ is at most $q-k$. Since $q-k +2< q$, by induction, we conclude that the $k$th entries of the $q+1$th row preserve the same A-sequence for $k \geq 3$. 
	
It can be seen that the matrix of the size $(q+1)\times (q+1)$ is uniquely determined. Now we need to check if the result holds for the two entries $B_{q+1,1}$ and $B_{q+1,2}$. Again, by the combinatorial interpretation, if we contract the objects of the entry $B_{q,1}$ into one, the entry $B_{q+1,1}$ can be interpreted as adding an object to the same box. Then it follows that $B_{q+1,1}=B_{q,1} \times b_2$. Similarly, the entry $B_{q+1,2}$ can be obtained by adding one object and one box to $B_{q,1}$ and adding one object into one existing box of $B_{q,2}$. It follows that $B_{q+1,2}=B_{q,1}+B_{q,2}\times b_2$.
	
Therefore, we conclude that the result indeed holds. Furthermore, if $b_p={b_2}^{p-1}$, the corresponding power series is indeed $\frac{t}{1-b_2t}$. 
\end{proof}

We now return to the problem raised in Example $1.1$. 

\begin{theorem}\label{thm:4.5}
Let $S$ be a subset of Appell subgroup ${{\cal A}}$ of the Riordan group ${{\cal R}}$ defined by $S=\{ (g,f): (g,f)\in {{\cal A}}, g'(0)\not= 0\}$. Then ${C}_{{\cal R}}(S)={{\cal A}}$. Particularly, if $(g,f)$ does not satisfy the condition $g'(0)\not=0$, then $C_{{\cal R}}$ contains elements that do not belong to ${{\cal A}}$. In general, $C_{{\cal R}}((g,t), g\in {\cal F}_0)=A^{\pm}$, the double Appell subgroup of ${{\cal R}}$.  
\end{theorem}

\begin{proof}
Let $(g,f)\in S\subset {{\cal A}}$ with $g'(0)\not= 0$. Then $(g,f)=(g,t)$ and $g=\sum_{k\geq 0}g_kt^k$ with $g_0, g_1\not= 0$. A Riordan array $(d,h)\in {{\cal C}}_{{\cal R}}(S)$ with $h=\sum_{k\geq 0} h_k t^k$ if and only if 
\[
(d,h)(g,t)=(g,t)(d,h);\,\, i.e.,\,\, (d(g\circ h), h)=(gd,h).
\]
Hence, $g(h)=g$. From \eqref{0.0}, we have

\begin{equation}\label{0.0-3}
g(h(t)=\sum_{n\geq 0}c_n t^n,
\end{equation}
where, by denoting ${\mathcal {C}}_{n}=\{(i_{1},i_{2},\dots ,i_{k})\,:\ 1\leq k\leq n,\ i_{1}+i_{2}+\cdots +i_{k}=n\}$, 

\begin{eqnarray*}
&&c_{n}=\sum _{\mathbf {i} \in {\mathcal {C}}_{n}}g_{k}h_{i_{1}}h_{i_{2}}\cdots h_{i_{k}}\nonumber\\
&=&g_1h_n+g_nh_1^n+\sum _{\mathbf {i} \in {\mathcal {C}}'_{n}}g_{k}h_{i_{1}}h_{i_{2}}\cdots h_{i_{k}},
\end{eqnarray*}
in which ${\mathcal {C}}'_{n}=\{(i_{1},i_{2},\dots ,i_{k})\,:\ 2\leq k\leq n-1,\ i_{1}+i_{2}+\cdots +i_{k}=n\}$. It is clear that the summation in the last equation contains terms involving only $h_k$ with $1\leq k\leq n-1$ because 

\[
i_1+i_2+\cdots i_k =n
\]
and $k\geq 2$ imply $1\leq i_j\leq n-1$ for all $j=1,2,\ldots k$. In addition, the summation contains terms involving $g_k$ with $2\leq k\leq n-1$, but not $g_1$ and $g_n$. For instance, 

\begin{eqnarray*}
&&c_1=g_1h_1,\\
&&c_2=g_1h_2+g_2h_1^2\\
&&c_3=g_1h_3+g_3h_1^3+2g_2h_1h_2,\,\ldots 
\end{eqnarray*}
Comparing the coefficients of the same powers of $g(h)=g$ yields 

\begin{equation}\label{0.0-4}
g_n=c_n=g_1h_n+g_nh_1^n+\sum _{\mathbf {i} \in {\mathcal {C}}'_{n}}g_{k}h_{i_{1}}h_{i_{2}}\cdots h_{i_{k}}.
\end{equation}
Therefore, for $n=1$, we have 

\[
g_1=g_1h_1.
\]
Since $g_1\not= 0$, we obtain $h_1=1$. Considering $n=2$, we have 

\[
g_2=c_2=g_1h_2+g_2h^2_1=g_1h_2+g_2,
\]
which implies $h_2=0$. Similarly, we have $h_3=0$. Assume that $h_k=0$ for $2\leq k\leq n-1$. Then from 

\begin{eqnarray*}
&&g_n=c_n=g_1h_n+g_nh_1^n+\sum _{\mathbf {i} \in {\mathcal {C}}'_{n}}g_{k}h_{i_{1}}h_{i_{2}}\cdots h_{i_{k}}\\
&=&g_1h_n+g_n
\end{eqnarray*}
we obtain $h_n=0$. Thus, $h=t$, which means ${C}_{{\cal R}}(S)={{\cal A}}$. Combining the above and Example 1.1, we complete the proof. 
\end{proof}

\medbreak
\noindent{\bf Remark 2.1} In Theorem \ref{thm:4.5}, the set $S$ may be a singleton set. From Example $1.1$, we may know that the condition $g'(0)\not= 0$ 
is necessary to insure that ${C}_{{\cal R}}(S)={{\cal A}}$. In addition, this condition can not be replaced to be $g(z)\not= 1$ since $C_{{\cal R}} (1+z^2,z)={{\cal A}}^{\pm}$, which is shown in Example $1.1$. For instance $(g,-z)\in {C}_{{\cal R}}(1+z^2,z)$. As examples of nontrivial $g$ with $g'(0)=0$, we may see $g(z)=(1-z+z^2)/(1-z)$, $(1-2z-\sqrt{1-4z})/(2z)$, etc.

\section{The subgroup of ${\cal F}_1$ and centralizers of Riordan group}\label{sec:2}

Let $h(z)=z(1+\sum_{i\geq 1} b_i z^i)$ and $g(z)=z(1+\sum_{i\geq 1}a_i z^i)$. Then 
$(g\circ h)(z)=z(1+\sum_{i\geq 1} d_i z^i)$. Successively carrying out multiplications and combining
similar terms, we obtain the following expressions for the canonical coordinates $d_i$:
\begin{eqnarray*}
&&d_1 = a_1 + b_1,  \quad (4_1)\\
&&d_2 = a_2 + b_2 + 2a_1b_1, \quad (4_2)\\
&&d_3 = a_3 + b_3+2a_1b_2 + 3a_2 b_1+ a_1b_1^2,\quad (4_3)\\
&&\cdots \cdots\\
&&d_n= a_n + b_n +\sum^{n-1}_{i=1}\alpha_i\phi_{n-i}(b_1,b_2,\ldots, b_{n-i}).\quad (4_n)
\end{eqnarray*}

The expressions $\phi_r(b_1,b_2, \ldots, b_r)$ are universal homogeneous polynomials of multidegree $r$ with integer coefficients. It can be seen that the formulae $(4_k)$, $k = 1,2,\ldots,n$, define a group structure on the direct product ${{\mathbb K}}^n$. This group will be denoted by ${\mathcal F}^{(n)}_1({{\mathbb K}})$. Clearly, ${\mathcal F}^{(1)}_1({{\mathbb K}})$ is an Abelian group isomorphic to the additive group of the ring ${{\mathbb K}}$. The group ${\mathcal F}^{(2)}_1({\mathbb K})$ is isomorphic to the additive group of ${{\mathbb K}}\oplus {{\mathbb K}}$, and the corresponding isomorphism depends non-linearly on the canonical coordinates $(b_1,b_2)$, etc. (see Babenko \cite{Bab}). 

\begin{proposition}
$({\cal F}_1,\circ)$ is a group with the identity $z$. For any $f\in {\cal F}_1$, the inverse of $f$ is its compositional inverse.
\end{proposition}

The properties of the substitution group of power series are discussed by Jennings \cite{Jen}, Johnson \cite{Joh}, and Babenko \cite{Bab}. 

We denote the centralizers of a subset $S\in {\cal F}_1$ as 
$C_{{\cal F}_1}(S)=\{ h\in{\cal F}_1:h\circ f=f\circ h \,\,\mbox{for all}\,\, f\in S\}.$ We now use the centralizers of ${\cal F}_1$ to describe some centralizers of ${\cal R}$. 

\begin{proposition}\label{pro:3.2}
The centralizer of a Lagrange type Riordan array $(1,h)$ with $h\in {\cal F}_1$, $h'(0)=1$, is 

\begin{equation}\label{4.1}
C_{{\cal R}}(1,h)=\left\{\begin{array}{ll} {\cal R} & if\, h=z\\ \{ (1,f):f\in C_{{\cal F}_1}(h)\} &if\, h\not= z.\end{array}\right.
\end{equation}
\end{proposition}

\begin{proof}
It is obvious that if $h=z$, then $C_{\cal R}(1,z)={\cal R}$. 

Suppose $h(z)=\sum_{n\geq 1} h_nz^n\not= z$, thus there exists a positive integer $\ell$ such that $h_\ell\not=0$ and $h_2=h_3=\cdots=h_{\ell-1}=0$. 
Without loss of generality, we may assume $h_1=1$. From compositional formula \eqref{0.0} of $g\in {\cal F}_0$, where $g(z)=\sum_{n\geq 0} g_nz^n$ with $g_0=1$, and the above $h\in {\cal F}_1$, we have $(g\circ h)(z)=\sum_{n\geq 0}d_nz^n$, where $d_0 = g_0=1$ and the other coefficient $c_n$ for $n \geq 1$ can be expressed as a sum over compositions of $n$  or as an equivalent sum over partitions of $n$:

\[
d_{n}=\sum _{\mathbf {i} \in {\mathcal {D}}_{n}}g_{k}h_{i_{1}}h_{i_{2}}\cdots h_{i_{k}},
\]
where

\[
{\mathcal {D}}_{n}=\{(i_{1},i_{2},\dots ,i_{k})\,:\ 1\leq k\leq n,\ i_{1}+i_{2}+\cdots +i_{k}=n\}
\]
is the set of compositions of $n$ with $k$ denoting the number of parts. If $(g,f)\in C_{\cal R}(1,h)$, then $(g,f)(1,h)=(1,h)(g,f)$ implies 

\begin{equation}\label{4.2}
h(f)=f(h)\quad \mbox {and}\quad  g=g(h),
\end{equation}
where the first equation of \eqref{4.2} implies $f\in C_{{\cal F}_1}(h)$, while the second equation derives $[z^n]g=[z^n]g(h)$, i.e., $g_n=d_n$, 
which implies

\begin{equation}\label{4.3}
g_n=\sum _{\mathbf {i} \in {\mathcal {D}}_{n}}g_{k}h_{i_{1}}h_{i_{2}}\cdots h_{i_{k}}.
\end{equation}
For the $h$ assumed above, i.e., $h(z)=\sum_{n\geq 1} h_nz^n\not= z$ with $h_\ell\not=0$ and $h_2=h_3=\cdots=h_{\ell-1}=0$, we substitute 
$n=\ell$ and the conditions of $h_{i_1}, h_{i+2}, \ldots,$ and $h_{i_k}$ into the \eqref{4.3} and get 

\[
g_\ell=g_1h_\ell+2 g_2h_1h_2+\cdots +g_\ell h_1^\ell.
\]
In the sum on the right-hand side of the above equation, every term except the first one and the last one contains $h_2, h_3, \ldots,$ and $h_{\ell-1}$. Thus the equation can be reduced to 

\[
g_\ell=g_1h_\ell+g_\ell h_1^\ell=g_1h_\ell+g_\ell,
\]
which implies $g_1=0$ due to $h_\ell\not= 0$. For $n=\ell$, \eqref{4.3} gives 

\[
g_{\ell+1}=g_1h_{\ell+1}+2g_2h_1h_\ell+\cdots+g_{\ell+1}h_1^{\ell+1}=2g_2h_\ell+g_{\ell+1}
\]
because $g_1=0$, $h_1=1$, and $h_2=h_3=\cdots=h_{\ell-1}=0$. Thus from the  assumption of $h_\ell\not= 0$, we obtain $g_2=0$. Similarly, by using the induction assumption of $g_2=g_3=\cdots =g_m=0$, we may have $g_{m+1}=0$ from 

\begin{eqnarray*}
&&g_{\ell+m}=g_1h_{\ell+k}+2g_2h_1h_{\ell+m-1}+\cdots\\
&&\quad +g_{m+1}((m+1)h_1^mh_\ell+\mbox{terms contain $h_2,h_3,\ldots, h_{\ell -1}$})+\cdots\\
&&\quad+g_{\ell+m}h_1^{\ell+m}.
\end{eqnarray*}
More precisely, from the induction assumption and the conditions of $h_2=h_3=\cdots=h_{\ell-1}=0$, the above equation can be reduced to 

\[
g_{\ell+m}=(m+1)g_{m+1}h_\ell+g_{\ell+m},
\]
which implies $g_{m+1}=0$. Hence, we have proved that for $h(z)\not= z$, every element $(g,f)$ of the centralizers of $(1,h)$ in Riordan group must has 
$g=1$, completing the proof of the proposition. 
\end{proof}

Similarly, we have the following result.
\begin{proposition}\label{pro:3.3}
The centralizer of a Bell type Riordan array $(h/z,h)$ is 

\[
C_{{\cal R}}(h/z,h)=\left\{\begin{array}{ll} {\cal R} & if\, h=z\\ \{ (f/z,f):f\in C_{{\cal F}_1}(h)\} &if\, h\not= z.\end{array}\right.
\]
\end{proposition}

\begin{corollary}
Let $S$, and $T$ be subsets of the Appel and Bell subgroups of ${\cal R}$, which contain more than one element. Then we have $C_{{\cal R}}(S)=\{ (1,f):f\in C_{{\cal F}_1}(h)\}$ and $C_{{\cal R}}(T)=\{ (f/z,f):f\in C_{{\cal F}_1}(h)\}$, respectively. 
\end{corollary}

\noindent{\bf Remark 3.1} Propositions \ref{pro:3.2} and \ref{pro:3.3} are not trivial. Although the maps $\gamma_1: {\cal L}\to {\cal F}_1$ and $\gamma_2: {\cal B}\to {\cal F}_1$ defined by $\gamma_1 (1,f)=f$ and $\gamma_2(f/z,f)=f$, respectively, for all $f\in{\cal F}_1$ are group isomorphisms, $\gamma_k$, $k=1,2$, do not transfer the centralizers from ${\mathcal F}_1$ to ${\mathcal{R}}$. This is because ${\mathcal F}_1$ contains only one type formal power series, $f\in \mathcal{F}_1$ with $f(0)=0$ and $f '(0)\not=0$, while Riordan array $(g,f)$ contains two type formal power series including the above $f$ and $g\in{\mathcal {F}}_0$ with $g(0)\not= 0$.

\section{Related topics}\label{sec:3}

In this section, first we recall the definition of a reverser of a group and study the collection of all reversers of a group as well as their applications to the centralizers of a group. The part on the reversers is a slight modification of a survey shown on pages 25-27 of O'Farrell and Short \cite{OFS}. Secondly, we consider the set of all elements of a group that can be reversed by an element of a group, particularly, the set of all elements of the Riordan group ${{\cal R}}$ that can be reversed by $M(1,-z)$, i.e., the set of pseudo-involutions of ${{\cal R}}$. Shapiro and one of the authors \cite{HS20} found that this set is a twisted subgroup of ${{\cal R}}$. 

\begin{definition}\label{def:1.2}
An element $g$ of a group $G$ is said to be \textit{reversible} in $G$ if there is another element $h$ of $G$ such that 
\[
hgh^{-1}=g^{-1}.
\]
We say that $h$ reverses $g$ or that $h$ is a \textit{reverser} for $g$. We denote by $R_g(G)$, or just $R_g$, the set of reversers of $g$. 

Denote by $R_G(h)$ the set of all elements in $G$ that can be reversed by $h$, namely,

\begin{equation}\label{1.4-2}
R_G(h)=\{ g\in G: hgh^{-1}=g^{-1}\}.
\end{equation}
We call $R_G(h)$ the reverse set in $G$ reversed by $h$. 
\end{definition}

\begin{definition}\label{def:1.3}
Two elements $f$ and $g$ of a group $G$ are said to be \textit{conjugate} in $G$ if there is a third element $h$ such that 
\[
f=h^{-1}gh,
\]
which can be denoted by $g^h$. Then $f$ is called the conjugate of $g$ by $h$. 
\end{definition}

\begin{definition}\label{def:1.4}
The commutator of two elements, g and h, of a group G, is the element
\[
[g, h] = g^{-1}h^{-1}gh.
\]
It is equal to the group's identity if and only if g and h commute. The subgroup of G generated by all commutators is called the derived group or the commutator subgroup of G. The commutator subgroup of the Riordan group ${{\cal R}}$ is discussed in Luz\'on, Mor\'on, and Prieto-Martinez \cite{LMPM}. 
\end{definition}

The relationship between the conjugate of $g$ by $h$ and the commutator of $g$ and $h$ is 
\[
g^h=g[g,h].
\]
Particularly, if $g\in C_G(h)$, then $[g,h]=e$ and $g^h=g$. 

\begin{proposition}\cite{OFS}
Let $g$ be reversible in $G$, and suppose that $h \in R_g$. Then \
\[
R_g = C_gh = hC_g.
\]
\end{proposition}
Here is an explanation of the last proposition. Let $f \in C_g$. Since $h \in R_g$, then
$fg=gf$ and $hgh^{-1}=g^{-1}$. We have 
\[
(hf)g(hf)^{-1} = hfgf^{-1}h^{-1} = hgff^{-1}h^{-1} = hgh^{-1}  = g^{-1}.
\]
Thus $hC_g \subset R_g$. Conversely, let $k \in R_g$, then
\begin{equation}
kgk^{-1}=g^{-1}.
\end{equation}
We have 
\[
(h^{-1}k)g(h^{-1}k)^{-1}  = h^{-1}kgk^{-1}h = h^{-1}g^{-1}h = g.
\]
Thus $h^{-1}k \in C_g$, which implies that $k \in hC_g$. Hence, $R_g \subset hC_g$. 
The same method could be applied to show $R_g=C_gh$. 	

The above proposition motivates the definition of the extended centralizer:
\[
E_g(G):=\{ h\in G: g^h=g\quad or \quad g^h=g^{-1}\} =C_g\cup R_g.
\]
This is also a subgroup of $G$. 

\begin{proposition}\cite{OFS}
Let $g\in G$. Then

(i) if $g$ is not reversible, then $E_g=C_g$ and $R_g=\phi$.

(ii) if $g$ is an involution, then $E_g=C_g=R_g$.
\end{proposition}

\begin{proposition}\cite{OFS}
Let $g$ and $h$ be elements of a group $G$. Then (i) $C_{g^{-1}}=C_g$; (ii) $R_{g^{-1}}=R_g$; (iii) $C_{hgh^{-1}}=hC_gh^{-1}$; (iv) $R_{hgh^{-1}}=hR_gh^{-1}$.
\end{proposition}

The extended centralizer group $E_g$ consists of $C_g$ and a single coset $R_g$. If $C_g$ is abelian, then we have a generalized dihedral group. Hence, there is a question arises: are there examples where $C_g$ is non-abelian? The answer is ''yes''. Here are two examples. If $G$ is non-abelian group with the identity $e$, then the centralizer $C_G(e)=G$ is not abelian. Another example is the dihedral group of order $4n$. This has an element, the rotation of $180^o$, in its center.

We now give the definition of a twisted subgroup.

\begin{definition}\label{def:3.1}
The set of elements reversed by $h\in{G}$ forms a twisted subgroup $T\equiv T_h:=\{ g\in G: hgh^{-1}=g^{-1}\}$ of $G$ defined by (see Foguel and Ungar \cite {FU00, FU01}): 

(i) $e\in T$ ($h$ reverses $e$)

(ii) $T$ is closed under taking inverse (if $h$ reverses $g$, then $h$ reverses $g^{-1}$).

(iii) if $x,y\in T$, then $xyx\in T$.
\end{definition}

The set $R(G)$ of reversible elements in $G$ is a union of twisted subgroups. $R(G)$ contains the identity and is closed under taking inverse, but it is not necessarily closed under composition. However, if $hfh^{-1}=f^{-1}$ and $hgh^{-1}=g^{-1}$, i.e., $f,g\in R(G)$ and both are reversed by the same $h$, then 
\[
(hg)(fg)(hg)^{-1}=(fg)^{-1},
\]
i.e., $hg\in R_{fg}$. 

Considering the set of the pseudo-involutions of Riordan group ${\cal R}$, which means the set of all $D\in {\cal R}$ such that $MD$ (and $DM$) is an involution, where $M=(1,-t)$. Then $MDM^{-1}=D^{-1}$. Shapiro \cite{Sha2} proved that the centralizer of $M$ in ${{\cal R}}$ is ${\mathcal C}$, the checkerboard subgroup (cf. also Jean-Louis and Nkwanta \cite{JLN}). Hence, from \cite{HS20} and \cite{Sha2} we have the the first half and the second half, respectively, of the following result. 

\begin{proposition}\label{pro:3.2-2}
The reserves set in ${{\cal R}}$ reversed by $M$, $R_{{\cal R}}(M)$, is a collection of all elements of pseudo-involutions form a twisted subgroup, which is denoted by $T_M$. The centralizers of $M$ in ${{\cal R}}$, 
$C_{{\cal R}}(M)$, is the checkerboard subgroup of ${{\cal R}}$. 
\end{proposition}

\begin{proof}
For the sake of the readers' convenience, we give a brief proof. First, $I=(1,t)\in T_M$ because $MIM^{-1}=I$. Secondly, $D=(d,h)$ and $E=(g,f)\in R_{{\cal R}}(M)$ implies that $DE^{-1}D\in R_{{\cal R}}(M)$ due to 

\begin{eqnarray*}
&&MDE^{-1}DM^{-1}=(MDM^{-1})(ME^{-1}M^{-1})(MDM^{-1})\\
&=&D^{-1}ED^{-1}=(DE^{-1}D)^{-1}.
\end{eqnarray*}
Hence, $R_{{\cal R}}(M)$ is the twisted subgroup $T_M$. 

If $D=(d,h)\in C_G(M)$, then $MD=DM$ implies that $(1,-t)(d,h)=(d,h)(1,-t)$, i.e., 

\[
(d(-t),h(-t))=(d, -h).
\]
Hence, $d(-t)=d(t)$ and $h(-t)=-h(t)$, i.e., $(d,h)$ is in the checkerboard subgroup of ${{\cal R}}$. 
\end{proof}

Let $G$ be a group and let $T\subset G$ be a twisted subgroup, from Lemma $1.2$ of Aschbacher \cite{Asc} we have $\langle x\rangle \subset T$ for each $x\in T$. Furthermore, if $T$ is a $2$-divisible twisted subgroup, i.e., for each $x\in T$, there exists a unique element $x^{1/2}\in T$ such that $(x^{1/2})^2=x$, from Corollary 3 of Glauberman \cite{Gla}, we have Lagrange's Theorem on twisted subgroup, i.e., for every $x\in T$, the order of $x$ divides $|T|$. Since $T_M$ is not finite, the Lagrange's Theorem is invalid on $T_M$. 

Shapiro and one of the authors \cite{HS20} presents a palindromic property of pseudo-involutions. More precisely, if $A$ and $B$  are both pseudo-involutions, then so is the triple product $ABA$. With this it follows that if are pseudo-involutions so is any palindromic word using these symbols. 

\section{acknowledgements}

The authors wish to express their sincere gratitude and appreciation to Professor Louis Shapiro, the referees, and the handling editor for their helpful comments and remarks that led to an revised version of the original manuscript. 

\end{document}